\documentclass[12pt]{article}
\usepackage{amssymb,amsmath,makeidx,amscd,verbatim,amsthm,graphicx,mathrsfs}
\usepackage{lscape}
\usepackage{bbm}
\usepackage{color}
\usepackage{hyperref}
\usepackage[english]{babel}
\addto{\captionsenglish}{%

}
\newtheorem{thm}{Theorem}[section]

\theoremstyle{definition}
\newtheorem{defx}[thm]{Definition}

\numberwithin{equation}{section}

\newcommand{\be}{\begin{enumerate}}
\newcommand{\ee}{\end{enumerate}}
\newcommand{\bq}{\begin{eqnarray*}}
\newcommand{\eq}{\end{eqnarray*}}

\begin{document}
\pagenumbering{arabic} \baselineskip 10pt
\newcommand{\disp}{\displaystyle}
\renewcommand*\contentsname{Table of Contents}
\thispagestyle{empty}
\newcommand{\HRule}{\rule{\linewidth}{0.1mm}}
\linespread{1.0}
\pagenumbering{arabic} \baselineskip 12pt
\thispagestyle{empty}
\title{\textbf{Equivalence of Quaternionic Heisenberg Homogeneous Quasi-norms}}
\author{O. A. Ariyo$^1$  and  M. E. Egwe$^{2*}$\\ Department of Mathematics, University of Ibadan, Ibadan, Nigeria.\\
$^1$\emph{oa.ariyo@ui.edu.ng}\;\;$^2$\emph{murphy.egwe@ui.edu.ng}}
\maketitle
\large	
\begin{abstract}
Let $\mathbb{H}_q$ denote the quaternionic Heisenberg group of dimension $(4n+3)$ with $\mathbb{R}^4\times\mathbb{R}^3$ stratification. We identify certain homogeneous norms on the group and show that any two quasi-norms on $\mathbb{H}_q$ are equivalent for $n<\infty$.
\end{abstract}


\section{Introduction}
Lie groups of $H$-type are generalization of the classical Heisenberg group. The Quaternionic Heisenberg group $\mathbb{H}_q$ is an example of a $H$-type group as introduced by Kaplan \cite{kaplan2006harmonic}. The group plays core roles in abstract harmonic analysis, the representation theory, analysis of several complex variables, the partial differential equations and quantum mechanics like its Heisenberg counterpart. It is a stratified Lie group with the underlying manifold structure $\mathbb{H}_q=\mathbb{H}\oplus \mathbb{R}^3\approx \mathbb{R}^4\times \mathbb{R}^3$, where $\mathbb{H}$ is the group of quaternions and isomorphic to $\mathbb{R}^4$. The multiplication is given by \[(u,v)(r,s)=(u+r, v+s+2\Im(r\cdot \bar{u}))~~~\hbox{where}~r\cdot \bar{u} =\sum_{j=1}^n r_j\overline{u_j}\]  $\forall~~u,r\in\mathbb{R}^4$ and $v,s\in\mathbb{R}^3$.\\

The centre of quaternionic Heisenberg group $\mathbb{H}_q$ is $\mathbb{R}^3=[\mathbb{H}_q,\mathbb{H}_q]$, and the bi-invariant Haar measure on $\mathbb{H}_q$ is the Lebesgue measure $dg:=dudt, ~\hbox{for}~u\in\mathbb{R}^4 ~\hbox{and} ~t\in\mathbb{R}^3$. Let $K$ be a complex compact subgroup of automorphism of $\mathbb{H}_q$, we define a motion group of semi-direct product of $\mathbb{H}_q$ and $K$ by $G:= \mathbb{H}_q\ltimes K$ with the usual product $\disp (k,x,t)\cdot(k^\prime, x^\prime, t^\prime)=[k\cdot k^\prime, (x,t)(k\cdot x^\prime,t)]$. The Haar measure on this motion group $G$ is $dudtdk$ where $dk$ is the Haar measure of $K$.\\
The Kohn-Laplacian operator is defined by $\Delta_{\mathbb{H}\times \mathbb{R}^3}=X_0^2+X_1^2+X_2^2+X_3^2$ \cite{faress2020spherical} and the hypoelliptic Sub-Laplacian is given by $\disp{\mathfrak{L}=-\frac{1}{4}\sum_{1\le j\le n,~~0\le k\le 3} \left(X_j^k\right)^2}$ such that
\begin{equation*}
\begin{split}
	-\Delta=&-\big(\frac{\partial^2}{\partial x_0^2}+\frac{\partial^2}{\partial x_1^2}+\frac{\partial^2}{\partial x_2^2}+\frac{\partial^2}{\partial x_3^2}\big)+4|x|^2\big(\frac{\partial^2}{\partial t_1^2}+\frac{\partial^2}{\partial t_2^2}+\frac{\partial^2}{\partial t_3^2}\big)\\
	&+\big(-x_1\frac{\partial}{\partial x_0}+x_0\frac{\partial}{\partial x_1}+x_3\frac{\partial}{\partial x_2}-x_2\frac{\partial}{\partial x_3}\big)T_1\\&+ \big(-x_2\frac{\partial}{\partial x_0}-x_3\frac{\partial}{\partial x_1}+x_0\frac{\partial}{\partial x_2}+x_1\frac{\partial}{\partial x_3}\big)T_2\\
	&+\big(-x_3\frac{\partial}{\partial x_0}+x_2\frac{\partial}{\partial x_1}-x_1\frac{\partial}{\partial x_2}+x_0\frac{\partial}{\partial x_3}\big)T_3,
	\end{split}
\end{equation*} see \cite{christ2016sharp}.\\
The basis for the Lie algebra of $\mathbb{H}_q$ is given by the horizontal left-invariant vector fields $X_0, X_1, X_2, X_3, T_1, T_2, T_3$ where \begin{equation*}
\begin{split}
X_0&=\frac{\partial}{\partial x_0}-2x_1T_1-2x_2T_2-2x_3T_3\\
X_1&=\frac{\partial}{\partial x_1}+2x_0T_1-2x_3T_2+2x_2T_3\\
X_2&=\frac{\partial}{\partial x_2}+2x_3T_1+2x_0T_2-2x_1T_3\\
X_3&=\frac{\partial}{\partial x_3}-2x_2T_1+2x_1T_2+2x_0T_3.
\end{split}
\end{equation*}
and
$\disp{T_1=\frac{\partial}{\partial t_1},	T_2=\frac{\partial}{\partial t_2}}$ and $\disp{T_3=\frac{\partial}{\partial t_3}}$\\
The Lie bracket defined on these vectors fields satisfies the following non-trivial commutation relations:
$$[X_0,X_1]=[X_3,X_2]=4T_1,~[X_0,X_2]=[X_1,X_3]=4T_2,~[X_0,X_3]=[X_2,X_1]=4T_3.$$

The group $\{\delta_\rho: 0<r<\infty\}$ of dilations defined on $\mathbb{H}_q$ is expressed as $\delta_\rho(u,t)=(\sqrt{\rho}u, \rho t)$ for every element $(u,t)\in \mathbb{H}_q$.
\section{The Homogeneous quasi-norms and Equivalence}
The quasi-norms on Quaternionic Heisenberg group are homogeneous norms and are compatible with the group's stratification. Moreover, these norms respect the non-Euclidean geometry and dilations on the group as well as the Carnot-Carath\'eodory metric resulting from the horizontal vector fields.
\begin{defx}
A quasi-norm on the quaternionic Heisenberg group is a function
\begin{equation}\label{2.1}
|\cdot|_{\mathbb{H}_q} : \mathbb{H}_q \longrightarrow [0,\infty)
\end{equation}
satisfying;
\begin{enumerate}
\item[(i)] $\disp{|\delta_\rho \nu|_{\mathbb{H}_q} = \rho^Q|\nu|_{\mathbb{H}_q}}$, $\rho>0$; where $Q$ is the degree of homogeneity
\item[(ii)] $\disp{|\nu|_{\mathbb{H}_q}\ge 0}$ and $\disp{|\nu|_{\mathbb{H}_q} =0\iff \nu=0}$ ~~(non-negativity)
\item[(iii)] $\disp{|\nu^{-1}|_{\mathbb{H}_q}=|\nu|_{\mathbb{H}_q}}$
\item[(iv)] $\disp{|\nu_1\nu_2|_{\mathbb{H}_q}\le K\left(|\nu_1|_{\mathbb{H}_q}+|\nu_2|_{\mathbb{H}_q}\right)},~K\ge 1$~~(quasi-triangle inequality)
\end{enumerate}
for all $\nu:=(u,t)\in \mathbb{H}_q$.
\end{defx}
Note that in quasi-norms, the triangle inequality property of norms is replaced with
\begin{equation}\label{2.2}
\|x\cdot y\|\le K(\|x\|+\|y\|)~~~ \text{for~ some}~ K>1
\end{equation}
where norm is implied when $K=1$. So we shall call \eqref{2.2} the quasi-triangle inequality.

\begin{defx}
Let \eqref{2.1} be a homogeneous quasi-norm on $\mathbb{H}_q$ and $\delta_\rho$ a dilation on $\mathbb{H}_q$. The quasi-norm \eqref{2.1} is said to be dilation invariant if $$\|\delta_\rho(q,t)\|_{\mathbb{H}_q}=\rho\|(q,t)\|$$
\end{defx}
Any norm on the Quaternionic Heisenberg group is homogeneous and of degree $Q=4n+6$ with respect to the dilation of the group, i.e., $\disp{|\delta_\rho \nu|_{\mathbb{H}_q}=\rho^Q|\nu|_{\mathbb{H}_q}}$ for any $\nu \in \mathbb{H}_q$ \cite{christ2016sharp}\cite{egwe2013equivalence}.\\
The quaternionic quasi-norms include;
\begin{enumerate}
\item The Kor\'anyi or the gauge norm is defined by $$\|(q,t)\|_{\mathbb{H}_q}=\left(|q|^4+|t|^2\right)^{1/4}.$$
Note that $\disp{|q|^2=\sum_{i =1}^n |q_{i}|^2}$ and $\disp{|q_i|}$ defines the classical quaternionic norm; $|t|=\left(\disp{\sum_{i=1}^3 t_i^2}\right)^{\frac{1}{2}}$ is the usual Euclidean norm on $t\in \mathbb{R}^3$. This Kor\'anyi-type norm is a homogeneous norm of degree $1$ in close relation to the dilation of the Quaternionic Heisenberg group defined earlier. It is known that this norm is smooth away from the origin and satisfies the following conditions;
    \begin{enumerate}
    \item $\disp{|(q,t)^{-1}|=|(q,t)|}$
    \item $\disp{|(q,t)|=0 \implies q=0,~t=0}$
    \end{enumerate}
    This norm satisfies the quasi-triangle inequality being symmetric and sub-additive up to a multiplicative constant.\\
    To see this, (a) is trivial since $(q,t)^{-1}=(-q,-t)$ and for (b), we show that $\disp \|(q,t)\cdot(q^\prime, t^\prime)\|\le K\left(\|(q,t)\|+\|(q^\prime, t^\prime)\|\right);~~K\ge 1$\\
    Recall that the product of any two elements $\nu, \nu^\prime \in\mathbb{H}_q$ is given by $$\nu\cdot \nu^\prime=(q,t)\cdot(q^\prime, t^\prime)=(q+q^\prime, t+t+2\Im (q\cdot \bar{q^\prime})).$$
   Then
    \begin{eqnarray*}
    \|(q,t)\cdot(q^\prime, t^\prime)\|&=& \|(q+q^\prime), t+t^\prime+2\Im(q\cdot \bar{q^\prime})\|\\
    &=&\left(\|q+q^\prime\|^4+\|t+t^\prime +2\Im(q\cdot\bar{q^\prime})\|^2\right)^{1/4}\\
    \end{eqnarray*}
     Note that $\disp |q+q^\prime|^4\le 8(|q|^4+|q^\prime|^4)$ and
    \begin{eqnarray*} |t+t+2\Im(q\cdot \bar{q^\prime})|^2&\le& 2\left(|t+t^\prime|^2+16|\Im(q\cdot \bar{q^\prime})|^2\right)\\
    &\le& 2\left(|t|^2+|t^\prime|^2+16|q|^2|q^\prime|^2\right)
    \end{eqnarray*}
    Hence, we have
    \begin{eqnarray*}
    \|(q,t)\cdot(q^\prime, t^\prime)\|^4&\le&K^\prime\left(|q|^4+|q^\prime|^4+|t|^2+|t^\prime|^2\right)\\
    &\le& K\left(\|(q,t)\|+\|(q^\prime, t^\prime)\|\right)
    \end{eqnarray*}

    \item The Folland-Stein Gauge which is equivalent to the Kor\'anyi norm is given by $\disp{\|(q,t)\|=\left(|q|^2 +|t|\right)^{1/2}}$ and is most adopted in the study of Hardy and Sobolev-type spaces \cite{folland2020hardy}. It differs from the Kor\'anyi norm only by scaling.
    \item Homogeneous quasi-norm defined as $\disp{\|(q,t)\|_\alpha=\left(|q|^\alpha +|t|^{\alpha/2}\right)^{1/\alpha};~\alpha > 0}$ coincides with the kor\'anyi norm if $\alpha =4$.
    \item The Box norm $\disp{\|(q,t)\|=\sqrt{|q|^2+|t|^2}}$ is a Euclidean-type norm on the quaternionic heisenberg group. This norm is nonhomogeneous and under dilation and is usually employed in geometric embedding.
    \item The $max-type$ norm which is defined by $\disp{\|(q,t)\|_{max}=\max \left(|q|, |t|^{1/2}\right)}$.
    \item The Carnot-Carath\'eodory distance. It is a bi-Lipschitz sub-Riemannian norm which is comparable to the Kor\'anyi norm and is defined via the length of horizontal curves by $\disp{d\left((q,t), (q^\prime, t^\prime) \right):=\inf \bigg{\{}\int_0^1|\dot \gamma(s)|~ds: \gamma (0)=(q,t), \gamma (1) =(q^\prime, t^\prime), ~\gamma ~\hbox{horizontal}\bigg{\}}}.$
\end{enumerate}
\begin{defx}
Any two quasi-norms $\|\cdot\|_u$ and $\|\cdot\|_v$ on Quaternionic Heisenberg group $\mathbb{H}_q:=\mathbb{H}\times\mathbb{R}^3$ are said to be equivalent if there exists constants $k_1, k_2>0$ such that $\disp{k_1\|(q,t)\|_u\le \|(q,t)\|_v \le k_2\|(q,t)\|_u,~~\forall~(q,t)\in\mathbb{H}_q.}$
\end{defx}
To prove equivalence of norms, for instance, the Kor\'anyi and the max-type quasi-norms are equivalent since we can find an upper and lower bounds as follows;\\
\begin{eqnarray*}\|(q,t)\|_{\mathbb{H}_q}=(|q|^4+|t|^4)^{1/4}&\le& \left(2 \max(|q|^4, |t|^2)\right)^{1/4}\\
&=&2^{1/4}\max\left(|q|, |t|^{1/2}\right)\\
&=& 2^{1/4}\|(q,t)\|_{\max}
\end{eqnarray*}
This defines the upper bound; and \\
\begin{eqnarray*}
\|(q,t)\|_{\mathbb{H}_q}=\left(|q|^4+|t|^2\right)^{1/4}&\ge& \left(\max (|q|^4, |t|^2)\right)^{1/4}\\
&=& \max (|q|, |t|^{1/2})\\
&=& \|(q,t)\|_{\max}
\end{eqnarray*}
defines the lower bound.
\\Hence, the equivalence is expressed as $\disp{\|(q,t)\|_{\max}\le\|(q,t)\|_K\le 2^{1/4}\|(q,t)\|_{\max}}.$
\begin{thm}
Let $|\nu|_{\mathbb{H}_{q_1}}$ and $|\nu|_{\mathbb{H}_{q_2}}$ be any two continuous homogeneous norms on $\mathbb{H}_q$ invariant under dilation. Then $|\nu|_{\mathbb{H}_{q_1}}$ and $|\nu|_{\mathbb{H}_{q_2}}$ are equivalent.
\end{thm}
\begin{proof}
The statement of the theorem implies that we seek constants $C_1, C_2> 0$ such that $\forall~\nu:=(u,t)\in \mathbb{H}_q$ we have $$C_1|\nu|_{\mathbb{H}_{q_1}}\le |\nu|_{\mathbb{H}_{q_2}}\le C_2|\nu|_{\mathbb{H}_{q_1}}.$$
Let $S_1$ and $S_2$ be unit spheres defined by $\disp S_1=\{(u,t)\in \mathbb{H}_q:|(u,t)|_{\mathbb{H}_{q_1}}=1\}$ and $\disp S_2=\{(u,t)\in \mathbb{H}_q:|(u,t)|_{\mathbb{H}_{q_2}}=1\}$.\\
The spheres so defined are compact in $\mathbb{H}_q\setminus \{(0,0)\}$ since $|(u,t)|_{\mathbb{H}_q}$ is continuous and positive away from zero. Now define $\disp \varphi: S_2\rightarrow [0,\infty)$ by $\disp \varphi\left((u,t)\right)=|(u,t)|^{\delta_\rho}=\rho^Q|(u,t)|,~~Q\ge 1$. Then by continuity property of the distance function $|(u,t)|_{\mathbb{H}_{q_2}}$ and compactness of $S_1$, $\varphi$ attains minimum on $S_1$ and maximum on $S_2$ denoted by $m$ and $M$ respectively. If we let $\rho:=|(u,t)|_{\mathbb{H}_{q_1}}$, we will have $\delta_\rho(u,t)\in S_1$, so that $\disp{|(u,t)|_{\mathbb{H}_{q_2}}=|\delta_\rho(u,t)|_{\mathbb{H}_{q_2}}=\rho|(u,t)|_{\mathbb{H}_{q_2}}}$\\
$\disp \implies~m|(u,t)|_{\mathbb{H}_{q_1}} \le |(u,t)|_{\mathbb{H}_{q_2}}\le M|(u,t)|_{\mathbb{H}_{q_1}};~~\forall~(u,t)\in \mathbb{H}\setminus (0,0)$.
\end{proof}
\begin{thm} The Box norm is a Euclidean-type norm on $\mathbb{H}_q$ and is nonhomogeneous with respect to the Quaternionic Heisenberg group dilation.\end{thm}

\begin{proof}
It suffices to show the non-homogeneity of this norm. To do this, we see by definition that
 \begin{eqnarray*}
\|\delta_\rho(q,t)\|&=&\sqrt{|\rho q|^2+|\rho^2t|^2}\\
&=&\sqrt{\rho^2|q|^2+\rho^4|t|^2}\\
&=& \sqrt{\rho^2\left(|q|^2+\rho^2|t|^2\right)}\\
&=& \rho\sqrt{|q|^2+\rho^2|t|^2}\ne \rho \|(q,t)\|.
\end{eqnarray*}
Therefore, the Box norm on $\mathbb{H}_q$ is non-homogeneous on $\mathbb{H}_q$.
\end{proof}

\bibliographystyle{amsplain}

\begin{thebibliography}{QHG Norms and their equivalence2Archive}
\bibitem{bonfiglioli2007stratified} Bonfiglioli, Andrea and Lanconelli, Ermanno and Uguzzoni, Francesco\ \emph{Stratified Lie groups and potential theory for their sub-Laplacians}, Springer Science \& Business Media, 2007.
\bibitem{brown2001lecture} Brown, Russell\ \emph{Lecture notes: harmonic analysis}. USA, Lexington: University of KentucNy.[cit. 2009Y04Y14]. Dostupn{\^e} z WWW: http://www. ms. uNy. edu/arbrown/courses/ma773/notes. pdf, 2001.
\bibitem{chang2006geometric} Chang, Der-Chen and Markina, Irina.\ \emph{Geometric analysis on quaternion H-type groups}. Springer, The Journal of Geometric Analysis, Vol 16, pages 265-294, 2006.
\bibitem{christ2016sharp} Christ, Michael and Liu, Heping and Zhang, An\ \emph{Sharp Hardy-Littlewood-Sobolev inequalities on quaternionic Heisenberg groups}, Elsevier, Nonlinear Analysis, Vol 130, pages 361-395, 2016.
\bibitem{egwe2013equivalence} Egwe, Murphy E.\ \emph{The Equivalence of certain norms on the Heisenber group}, Adv. Pure Math, Vol 3, No 6, pages 576-578, 2013.
\bibitem{faress2020spherical} Faress, Moussa and Fahlaoui, Said\ \emph{Spherical Fourier transform on the quaternionic Heisenberg group}, Taylor \& Francis, Integral Transforms and Special Functions, Vol 31, No 9, pages 685-701, 2020.
\bibitem{folland2016course} Folland, Gerald B.\ \emph{A course in abstract harmonic analysis}, CRC press, 2016.
\bibitem{folland2020hardy} Folland, Gerald B and Stein, Elias M\ \emph{Hardy spaces on homogeneous groups}, Princeton University Press, Vol 107, 2020.
\bibitem{kaplan2006harmonic} Kaplan, A and Ricci, F\ \emph{Harmonic analysis on groups of Heisenberg type}, Springer, Harmonic Analysis: Proceedings of a Conference Held in Cortona, Italy, July 1-9, 1982, pages 416-435, 2006.
\bibitem{mauceri2006harmonic} Mauceri, Giancarlo and Ricci, Fulvio and Weiss, Guido\ \emph{Harmonic analysis: proceedings of a conference held in Cortona, Italy, July 1-9, 1982}, Springer, Vol 992, 2006.
\bibitem{yang2022harmonic} Yang, Zhipeng\ \emph{Harmonic analysis on 2-step stratified Lie groups without the Moore-Wolf condition}, 2022.
\end{thebibliography}

\providecommand{\bysame}{\leavevmode\hbox to3em{\hrulefill}\thinspace}

\end{document}